\documentclass[letterpaper]{article}

\usepackage[letterpaper]{geometry}
\usepackage{natbib}
\usepackage{amsmath}
\usepackage{amssymb}
\usepackage{dsfont}
\usepackage{amsthm}
\usepackage[shortlabels]{enumitem}
\usepackage{graphicx}
\usepackage{epstopdf}
\usepackage{url}
\usepackage{color}
\usepackage{mathtools}
\usepackage{tikz}
\usetikzlibrary{shapes.misc,decorations.pathreplacing,decorations.markings,patterns,arrows.meta}

\newcommand{\RR}{{\mathbb R}}
\newcommand{\BB}{{\mathbb B}}

\newcommand{\cI}{{\mathcal{I}}}

\newcommand{\cD}{{\mathcal{D}}}
\newcommand{\cM}{{\mathcal{M}}}

\newcommand{\cG}{{\mathcal{G}}}
\newcommand{\cS}{{\mathcal{S}}}
\newcommand{\cU}{{\mathcal{U}}}
\newcommand{\cV}{{\mathcal{V}}}

\newtheorem{theorem}{Theorem}
\newtheorem{proposition}{Proposition}
\newtheorem{remark}{Remark}

\newtheorem{problem}{Problem}
\newtheorem{lemma}{Lemma}

\newtheorem{assumption}{Assumption}

\title{Guaranteed Reachability for Systems with Unknown Dynamics}
\author{Melkior Ornik\thanks{University of Illinois at Urbana-Champaign, Urbana, IL 61801, USA. \quad e-mail: mornik@illinois.edu}}
\date{}

\begin{document}

\maketitle

\begin{abstract}
The problem of computing the reachable set for a given system is a quintessential question in nonlinear control theory. While previous work has yielded a plethora of approximate and analytical methods for determining such a set, these methods naturally require the knowledge of the controlled system dynamics throughout the state space. In contrast to such classical methods, this paper considers the question of estimating the reachable set of a control system using only the knowledge of local system dynamics at a single point and a bound on the rate of change of dynamics. Namely, motivated by the need for safety-critical planning for systems with unknown dynamics, we consider the problem of describing the guaranteed reachability set: the set of all states that are guaranteed to be reachable regardless of the true system dynamics, given the current knowledge about the system. We show that such a set can be underapproximated by a reachable set of a related known system whose dynamics at every state depend on the velocity vectors that are guaranteed to all control systems consistent with the assumed knowledge. Complementing the theory, numerical examples of a single-dimensional control system and a simple model of an aircraft in distress verify that such an underapproximation is meaningful in practice, and may indeed equal the desired guaranteed reachability set.
\end{abstract}

\section{Introduction}
Following an adverse event affecting flight safety, such as partial loss of actuation, significant change in system dynamics due to damage, or sensor malfunction, a pilot operating an aircraft in distress needs to determine whether to divert to an alternate airport, and if so, which airport to choose. While it \textit{may} be possible for the aircraft to continue to its original destination, the pilot is required to determine a diversion airport where the aircraft can \textit{certainly} land \citep{EksPan03,DeS13}. At the time of decision, the pilot might not have a correct model of flight dynamics; in extreme cases such as loss of a wing \citep{Alo06}, there may be little prior experience or available knowledge about the system. Thus, it is imperative to immediately provide the pilot, or the flight controller for an autonomous vehicle, with a set of landing points that the aircraft is guaranteed to be able to reach, given the current information about the system.

Abstracted from its motivation, the above example describes the classical problem of \textit{reachability}: given the system's initial state, we wish to describe all states for which there exists a control input that drives the system to that state \citep{Bro76,Isi85,Lew12}. Such a problem has been explored in many related flavors, such as reachability at a given time \citep{Dinetal11}, reachability within a given time interval \citep{Mitetal05}, reachability with control constraints \citep{KurVar00}, and reachability with state constraints \citep{KurVar06}. While determining the reachable set of a known system --- and even the reduced problem of deciding whether the reachable set equals the entire state space --- was shown to be generally computationally infeasible \citep{Son88,Kaw90}, there has been substantial work on providing approximations of the reachable set and computing it with arbitrary precision using iterative methods \citep{Wol90,MitTom03,Altetal08, Xueetal17}.

The crucial difference between the above work and the focus of this paper, motivated by the example of an aircraft in sudden distress, is in our assumption of a significant lack of knowledge on system dynamics when computing the reachable set. While previous work has considered computation of reachable sets under uncertainties in dynamics such as those given by bounded unknown disturbances \citep{Mitetal05} or a finite number of uncertain parameters \citep{Altetal08}, the framework that we are considering contains substantially fewer information. Namely, motivated by the work of \cite{Ornetal17,Ornetal19} that determines \textit{local controlled dynamics} of a nonlinear system at a given state using solely the information from a single trajectory until the time that the system reached that state, we assume that the only available knowledge at the time of computing the reachable set consists of (i) local dynamics at a single point and (ii) Lipschitz bounds on the rate of change of system dynamics in the state space. As we do not know anything else about the system dynamics, we wish to determine the set of states that are guaranteed to be reachable from that single point \textit{regardless} of the true system dynamics, as long as they are consistent with the above knowledge.

The primary contribution of this paper is to provide an underapproximation of such a \textit{guaranteed reachability set (GRS)}. Our approach for determining this underapproximation relies on interpreting a control system as a differential inclusion with the \textit{available velocity set} at every state in the state space describing the possible tangents to the system's trajectory when leaving that state. Local dynamics at each state determine the available velocity set at that state, and, while exact available velocity sets are not known anywhere except for a single point, we can determine the family of all velocity sets that are consistent with our knowledge about the system dynamics. The intersection of the elements of such a family provides the \textit{guaranteed velocity set} at every state in the state space; a set of velocities for which we know there exists a control input generating them at a given state, even if we do not know the exact dynamics at that state. Since such a set may be difficult to compute or express in closed form, we underapproximate it by a ball of maximal possible radius. Moving back from differential inclusions to controlled differential to differential equations, such an underapproximation generates a control system with \textit{known} dynamics; the reachable set of such a control system is a subset of the GRS.

The outline of this paper is as follows. In Section \ref{motiv} we connect our work with previous efforts on reachability and estimation of dynamics, discuss the realism of the knowledge that we assume about the system, and provide a broad discussion of the paper's approach and results. Section \ref{probst} provides the formal problem statement and assumptions about the system dynamics and prior available knowledge. Section \ref{gvel} introduces the guaranteed velocity set, expresses it as an intersection of a family of simply computable sets, provides its underapproximation by a ball, and shows that such an underapproximation is maximal. Section \ref{reachse} describes an underapproximation of the GRS as a reachable set of a known control system, discusses the currently available methods of computing such a set, and proposes a strategy for determining, in real time, a control input to drive the system to a state guaranteed to be reachable. In Section \ref{exa} we provide two numerical examples to illustrate and validate the developed theory. Section \ref{exa1} considers a nonlinear 1D control system motivated by a phenomenon of gene autoregulation and discusses the relationship between the true reachable set, the guaranteed reachable set, and its previously described underapproximation; its results serve to both confirm the preliminary theory established in this paper and inform future avenues of research. Section \ref{exa2} builds upon the motivating example by discussing a simple model of an aircraft in need to divert. Finally, Section \ref{conc} discusses open questions for future work.

\textbf{Notation.} The set of all matrices with $n$ rows and $m$ columns is denoted by $\RR^{n\times m}$. Notation $\BB^m(x;r)$ denotes a closed ball in $\RR^m$ with the center at $x\in\RR^m$ and radius $r\geq 0$. Set $GL(n)$ denotes the general linear group of all matrices $\RR^{n\times n}$, i.e., all invertible $n\times n$ matrices. For a vector $v$, $\|v\|$ denotes its Euclidean norm. For a matrix $M$, $M^T$ denotes its transpose and $\|M\|$ denotes its $2$-norm: $\|M\|=\max_{\|v\|=1}\|Mv\|$. For a matrix $M$, vector $v$, and set $\cS$ of vectors of appropriate size, notation $v+M\cS$ denotes the set $\{v+Ms~|~s\in\cS\}$. Notation $\mathrm{diag}(\lambda_1,\ldots,\lambda_n)$ denotes a diagonal $n\times n$ matrix with elements $\lambda_1$, \ldots, $\lambda_n$ on the diagonal, in that order. Vector $e_i\in\RR^n$ denotes the $i$-th coordinate vector, i.e., a vector of all zeros, except a $1$ in position $i$. Matrix $I$ denotes the identity matrix.

\section{Motivation and Overview}
\label{motiv}

The assumptions and objectives of this work come from the context of online learning and control of a nonlinear system with unknown dynamics. Such a framework, described by \cite{Ornetal17,Ornetal19}, is motivated by the desire to successfully control a system with entirely unknown dynamics by learning as much as possible about the dynamics ``on the fly'', i.e., solely from the system's behavior during a single system run; it differs from classical work on adaptive or robust control \citep{IoaSun96,DulPag00} by not assuming almost any knowledge about the magnitude or the structure of uncertainty about the system dynamics, and from work on data-driven learning and control synthesis by not allowing collection of information on system dynamics by way of repeated system runs \citep{Bruetal16,Cheetal18}. 

Under some technical assumptions, the results of \cite{Ornetal17,Ornetal19} show that, for an unknown control-affine system $\dot{x}=f(x)+G(x)u$, the \textit{local dynamics} $u\mapsto f(x_0)+G(x_0)u$ can be obtained for any state $x_0$ lying on a system trajectory, with an arbitrarily small error and using only the knowledge of the trajectory prior to the time at which $x_0$ is visited, as well as the bounds on Lipschitz constants and magnitudes of functions $f$ and $G$. Thus, along with the assumption of perfect observations of system state at every time, the only ``outside'' knowledge required to establish these local dynamics are the bounds on Lipschitz constants and magnitudes of functions $f$ and $G$. These bounds are not required to be sharp, and may thus come from basic knowledge of the physical laws and the system's environment.

The work of \cite{Ornetal17,Ornetal19} considers a predetermined control objective, i.e., a target state that should be reached. The proposed approach is to design a utility function which, when maximized, appears to lead the system in a desired direction, given the learned local system dynamics. However, a control signal driving the system to a desired state may not exist and, even if it does, it may be impossible to learn online. Consequently, this paper aims to take a different approach. While it may be impossible to know whether a system can be driven to a desired state, the learned dynamics, coupled with the bounds on the rate of change of the underlying vector fields, naturally yield the two following sets:

\begin{enumerate}[(i)]
\item states that \textit{may} be possible to reach using admissible control signals \textit{(optimistic reachability set)},
\item states that are \textit{guaranteed to be} reachable using admissible control signals \textit{(guaranteed reachability set --- GRS)}.
\end{enumerate}

The optimistic reachability set and the GRS are clearly a superset and a subset, respectively, of the true reachable set. While computing over- and underapproximations of the reachable set has been the subject of substantial previous research \citep{KurVar00,MitTom03}, including the setting of computing the optimistic reachability set for systems with uncertain parameters \citep{Altetal08}, prior work has not --- to the best of our knowledge --- considered the framework of an unknown control system with knowledge of local dynamics at a single state.

Without discussing the reachable sets, the work of \cite{Ornetal17,Ornetal19} has implicitly concentrated on optimistic reachability, i.e., attempting to reach the original objective while there appears to be any chance of reaching it. The contribution of this paper is to consider the latter option, that is, to provide a set of states that --- using the knowledge of local dynamics at the system's intital state and Lipschitz constants of the underlying vector fields --- are \textit{certifiably reachable} using an admissible control signal. While the motivation for our work stems from \cite{Ornetal17,Ornetal19}, our contribution is entirely standalone: we do not concern ourselves with how local dynamics were found, but assume that they were obtained, and seek to exploit them to find a reachable set of states.

Motivated by considerations similar to the ones of this paper, the work of \cite{Ahmetal17} follows the same broad idea of guaranteed reachability, deeming an unknown dynamical system ``safe'' if it is guaranteed that its trajectory will not enter a particular unsafe set. However, unlike our paper, the dynamical system in that work has no control input. Hence, the reachable set is solely a single trajectory; the focus of \cite{Ahmetal17} is thus on certifying the safety of a dynamical system. Our results focus on finding a certifiably reachable set for a control system; applied to the narrative of \cite{Ahmetal17}, our paper provides a certificate of existence of a control input that renders the resulting dynamical system safe.

The key idea of this paper is the following: if dynamics $u\mapsto f(x_0)+G(x_0)u$ are known at an initial state $x_0$ for the set of admissible controls $\cU$, then the set of system velocities $\cV_{x_0}$ available at $x_0$ is also known: $\cV_{x_0}=\{f(x_0)+G(x_0)u~|~u\in\cU\}=f(x_0)+G(x_0)\cU$. While $\cV_x$ may not be known for any other $x$, $x\mapsto\cV_x$ is a Lipschitz continuous function under some appropriate metric, given the Lipschitz bounds on $f$ and $g$. In other words, while we may not know $\cV_x$, we can compute the bound on the difference between $\cV_{x}$ and $\cV_{x_0}$ dependent on the distance $\|x_0-x\|$ and Lipschitz bounds on $f$ and $g$. Consequently, we are able to determine the set of \textit{guaranteed velocities} $\cV^\cG_{x}$, i.e., the intersection of all sets that can possibly equal $\cV_x$. Hence, regardless of the true system dynamics --- as long as they are consistent with our prior knowledge about the dynamics --- the set of velocities $\cV^\cG_x$ will be available at $x$. Such a set will satisfy $\cV^\cG_{x}\to\cV_{x_0}$ as $x\to x_0$, with some appropriate notion of convergence.

As $\cV^\cG_{x}\subseteq\cV_{x}$ for all $x$, any trajectory that satisfies the differential inclusion $\dot{x}\in\cV^\cG_x$ for all $x$ will be a trajectory of $\dot{x}=f(x)+G(x)u$ produced by some admissible control signal; an analogous claim holds for all other dynamics $\dot{x}=\hat f(x)+\hat{G}(x)u$ consistent with the assumed knowledge. Hence, the reachable set of the differential inclusion $\dot{x}\in\cV^\cG_x$, with $x(0)=x_0$, is a subset of the \textit{guaranteed reachablility set} --- the intersection of reachable sets for all control systems consistent with the assumed knowledge. As the reachable set of $\dot{x}\in\cV^\cG_x$ may be difficult to characterize analytically, we will also provide a method of its underapproximation by a reachability set of a simpler control system, based on underapproximating each set $\cV^\cG_x$ by a ball of maximal radius.

The remainder of this paper serves to describe, formalize, and quantify the notions outlined in the above paragraphs. We begin by posing a formal problem statement.

\section{Problem Statement}
\label{probst}

Throughout the paper, we consider a control system $\cM(f,G)$ given by the following control-affine system dynamics: 
\begin{equation}
\label{consys}
\begin{split}
& \dot{x}(t)=f(x(t))+G(x(t))u(t)\textrm{,} \\
& x(0)=x_0\textrm{,}
\end{split}
\end{equation}
where $t\geq 0$, $x(t)\in\RR^n$ for all $t$, admissible control inputs $u(t)$ lie in the set $\cU\subseteq\RR^m$, and functions $f:\RR^n\to\RR^n$ and $G:\RR^n\to\RR^{n\times n}$ are globally Lipschitz-continuous, i.e., there exist $L_f\geq 0$ and $L_G\geq 0$ such that $\|f(x)-f(y)\|\leq L_f\|x-y\|$ and $\|G(x)-G(y)\|\leq L_G\|x-y\|$ for all $x,y\in\RR^n$. While the framework of this paper, and many of its results, could be analogously applied to control systems on more general state spaces, we consider the Euclidean state space for ease of exposition.

Without loss of generality, we assume $x_0=0$; the system state can be shifted by $-x_0$ to ensure such a property. For technical reasons we also make the following assumption.

\vskip 5pt

\begin{assumption}
\label{ass}
The system described by \eqref{consys} is fully actuated at $x_0=0$, i.e., $m=n$ and $G(0)\in GL(n)$. Set $\cU$ satisfies $\cU=\BB^n(0;1)$.
\end{assumption}

The assumption of full actuation makes estimation of available system velocities at points $x$ around $x_0=0$ significantly simpler, as the estimates will depend on $\|G(0)^{-1}\|$. We will briefly discuss the case of a non-invertible $G(0)$ later in the paper: such an estimation can still be performed by using the reciprocal of the smallest non-zero singular value of matrix $G(0)$ instead of $\|G(0)^{-1}\|$. The assumption of $\cU=\BB^n(0;1)$ parallels an assumption of \cite{Ornetal17,Ornetal19} that $\cU=[-1,1]^m$; in our motivating example, the inputs available to the aircraft are clearly bounded, and taking $\cU=\BB^n(0;1)$ renders technical work significantly easier.

Building upon the assumptions and results of \cite{Ornetal17,Ornetal19}, we place the following limitations on our knowledge about the system dynamics: we assume that bounds $L_f$ and $L_G$ are known, along with values $f(0)$ and $G(0)$. We assume that nothing else is known about functions $f$ and $G$. 

\vskip 5pt

\begin{remark}
It can be trivially shown that knowing the values $f(0)$ and $G(0)$ is equivalent to knowing the local system dynamics $u\mapsto f(0)+G(0)u$.
\end{remark}

We will denote the set of all functions $(\hat{f},\hat{G})$ consistent with the above knowledge by $\cD_{con}$, by defining
\begin{equation*}
\begin{split}
\cD_{con} =\{(\hat{f},\hat{G})~|~\hat{f}(0)=f(0)\textrm{, }\hat{G}(0)=G(0)\textrm{, } \\ L_f \textrm{ is a Lipschitz bound for }\hat{f}\textrm{, } \\ L_G \textrm{ is a Lipschitz bound for }\hat{G}\}\textrm{.}
\end{split}
\end{equation*}

Our goal is to describe the set of states that are reachable from $x_0=0$, \textit{regardless} of the true dynamics of system \eqref{consys}, as long as they are consistent with the assumed knowledge. Computing this set, or its underapproximation, serves to enable the planner to determine an objective that is guaranteed to be reachable regardless of the unknown system dynamics. 

We define the \textit{(forward) reachable set} $R^{\hat f,\hat G}(T;x_0)$ as the set of all states reachable by a system $\cM(\hat f,\hat G)$ from $x_0$ at time $T$ using some control signal $u:[0,T]\to\cU$. In other words, if $\phi^{\hat f,\hat G}_u(T;x_0)$ denotes a controlled trajectory of system $\cM(\hat f,\hat G)$ with control signal $u$, 
\begin{equation*}
R^{\hat f,\hat G}(T,x_0)=\{\phi^{\hat f,\hat G}_u(T;x_0)~|~u:[0,T]\to\cU\}\textrm{.}
\end{equation*}
Throughout the paper, we assume that the existence and uniqueness of all trajectories $\phi_u(T;x_0)$ are guaranteed. We can now state the following central problem of this paper.

\vskip 5pt

\begin{problem}
\label{pro}
Let $T\geq 0$. Describe the \emph{guaranteed reachability set (GRS)} defined by
\begin{equation}
\label{grs}
R^\cG(T,0)=\bigcap_{(\hat{f},\hat{G})\in\cD_{con}} R^{\hat f,\hat G}(T,0)\textrm{.}
\end{equation}
\end{problem}

While we pose Problem~\ref{pro} in terms of finding the GRS at time $T$, we could analogously ask for the set of guaranteed reachability before time $T$, i.e., $\cup_{t\in[0,T]} R^\cG(t,0)$, and the set of eventual reachability $\cup_{t\in[0,+\infty)} R^\cG(t,0)$. Much of the theory presented in subsequent sections of this paper can be adapted to these two sets.

In our running example of an aircraft in distress, Problem \ref{pro} is only the first of two steps towards managing the crisis: after determining which airport to land at, the pilot needs to determine in real time \textit{how} to land at that airport. The interest of this paper is primarily in solving Problem~\ref{pro}. However, Section~\ref{reachse} also provides a method of determining, under technical assumptions, a control input to reach the desired state $x_{end}$ if $x_{end}$ is certified to be reachable. As such a control input will depend on system dynamics, the proposed method exploits previously mentioned work in \cite{Ornetal17,Ornetal19} on real-time estimation of local system dynamics.

Naturally, in order for Problem \ref{pro} to be meaningful, we wish to describe $R^\cG(T,0)$ as simply as possible. The primary contribution of this paper is to provide an underapproximation of $R^\cG(T,0)$ as a reachable set of the control system $\dot{x}=a+g(\|x\|)u$, where $g:[0,+\infty)\to[0,+\infty)$ is a ramp function. We begin with a discussion of a relationship of controlled system dynamics given by an ordinary differential equation and its induced differential inclusion.

\section{Guaranteed Velocities}
\label{gvel}

Ordinary differential equations with control inputs have been interpreted as ordinary differential inclusions in numerous previous works; see, e.g., \cite{Cla83,KurVar00,BrePic07}. Following the same approach, we define a differential inclusion
\begin{equation}
\label{incsys}
\begin{split}
\dot{x}\in \cV_x & =f(x)+G(x)\cU\textrm{,} \\
& x(0)=x_0\textrm{.}
\end{split}
\end{equation}
where in the future we denote $\cV_x=f(x)+G(x)\cU$. Clearly, any solution $\phi_u(\cdot;x_0)$ to the differential equation \eqref{consys} also satisfies \eqref{incsys}, while any trajectory that satisfies \eqref{incsys} at all times $t$ by definition satisfies \eqref{consys}. Let us define the reachable set of a differential inclusion \eqref{incsys} as a set of all possible values, at a given time, of trajectories that satisfy \eqref{incsys}. Then, the reachable set of \eqref{incsys} is naturally $R(T,x_0)$.

Given our assumptions about the available knowledge of the system, set $\cV_{x_0}=\cV_0$ is assumed to be entirely known, but no other set $\cV_x$ is known. Our goal in this section is to provide an underapproximation for the sets $\cV_x$ using solely the available knowledge of $\cV_0$, as well as Lipschitz constants $L_f$ and $L_G$.

Set $\cV_x$ describes all the available velocities for a control system that satisfies \eqref{consys}, i.e., \eqref{incsys}, at a time $t$ when its state is $\phi(t;x_0)=x$. If $f(x)$ and $G(x)$ were known, computing such a set would be simple. However, by our assumptions, we do not know their exact values, but do know that $(f,G)\in\cD_{con}$. Then, analogously to the GRS, we can define the \textit{guaranteed velocity set}, i.e., the set of all velocities that will certainly be available at $x$, regardless of which element in $\cD_{con}$ represents the true system dynamics. Such a set is given by 
\begin{equation}
\label{vgx}
\cV^{\cG}_x=\bigcap_{(\hat{f},\hat{G})\in\cD_{con}}\hat{f}(x)+\hat{G}(x)\cU\subseteq \cV_x\textrm{.}
\end{equation}

Let us first relate the reachable set of the differential inclusion 
\begin{equation}
\label{vdifinc}
\begin{split}
\dot{x}\in\cV^{\cG}_x\textrm{,} \\
x(0)=0\textrm{,}
\end{split}
\end{equation}
to $R^\cG(T,0)$.

\vskip 5pt
\begin{proposition}
\label{subsv}
Let $T\geq 0$. If a function $\phi:[0,+\infty)\to\RR^n$ satisfies \eqref{vdifinc} at all times $t\leq T$, then $\phi(T)\in R^\cG(T,0)$.
\end{proposition}
\begin{proof}
If $d\phi(t)/dt\in\cV^{\cG}_{\phi(t)}$, then $\phi$ satisfies $\dot{x}\in \hat{f}(x)+\hat{G}(x)\cU$ for all $(\hat{f},\hat{G})\in\cD_{con}$. Thus, by a discussion analogous to that under \eqref{incsys}, $\phi(T)\in R^{\hat{f},\hat{G}}(T,0)$ for all $(\hat{f},\hat{G})\in\cD_{con}$. By \eqref{grs}, $\phi(T)\in R^\cG(T,0)$.
\end{proof}

We note that set $\cV^\cG_x$ may be empty for certain $x$. Such a property does not disable us from discussing the reachable set of \eqref{vdifinc}; if a trajectory reaches a state $x$ where $\cV^\cG_x=\emptyset$, we will consider by convention that it ceases to exist at that time.

\begin{remark}
Proposition~\ref{subsv} states that the reachable set of \eqref{vdifinc} may be used as an underapproximation of $R^\cG(T,0)$. We do not claim that those sets are in general equal: the intersection of reachable sets of $\dot{x}\in F_i(x)$ may not equal the reachable set of $\dot{x}\in \cap_i F_i(x)$. To illustrate this fact, consider $F_1(x)=\{k[1 \textrm{ } 2]^T~|~k\in[0,1]\}$ and $F_2(x)=\{k[1\textrm{ } x_2]^T~|~k\in[0,1]\}$ with $x(0)=(0,1)$. Then, $\dot{x}\in F_1\cap F_2$ gives $\dot{x}=0$ at $x(0)$, i.e., the reachable set of $\dot{x}\in F_1\cap F_2$ at any time is just $\{x(0)\}$. On the other hand, the reachable set of $\dot{x}\in F_1(x)$ at time $T$ includes the state $(T,2T+1)$, and the reachable set of $\dot{x}\in F_2(x)$ at time $T$ includes the state $(T,e^T)$. Thus, at time $T$ that satisfies $e^T=2T+1$, the intersection of reachable sets of $\dot{x}\in F_1(x)$ and $\dot{x}\in F_2(x)$ contains at least one state other than $\{x_0\}$.

While this paper contains no theoretical discussion of equality between the reachable set of \eqref{vdifinc} and $R^\cG(T,0)$, the example in Section \ref{exa1} shows that the two sets may in fact be equal --- the above counterexample may not be applicable as $R^\cG(T,0)$ is not an intersection of any general reachable sets, but of reachable sets $R^{\hat{f},\hat{G}}(T,0)$, where $\hat{f}$ and $\hat{G}$ satisfy the additional structure $(\hat{f},\hat{G})\in\cD_{con}$.
\end{remark}

Motivated by the result of Proposition~\ref{subsv}, we seek to underapproximate $R^\cG(T,0)$ by considering the reachable set of \eqref{vdifinc}. To do that, let us first discuss the geometric properties of $\cV^{\cG}_x$. For a given $x\in\RR^n$ it is simple to show that 
\begin{equation}
\label{dcon2}
\{(\hat{f}(x),\hat{G}(x))~|~ (\hat{f},\hat{G})\in\cD_{con}\}= \BB^n (f(0);L_f\|x\|)\times\BB^n (G(0);L_G\|x\|)\textrm{.}
\end{equation}
Consequently, since $\cU=\BB^n(0;1)$ and, for a matrix $M$, $M\BB^n(0;1)$ is an ellipsoid \citep{Bar02}, $\cV^\cG_x$ is an intersection of a parametrized, infinite set of ellipsoids. Unfortunately, such an intersection may be difficult to express in a simple form; as noted by \cite{KurVar06b}, an intersection of any number of ellipsoids is generally not an ellipsoid. In particular, Figure~\ref{intell} shows that, while easy characterization may be possible in special cases such as when $G(0)=I$, $\cV^{\cG}_x$ cannot be expected to be an ellipsoid.

\begin{figure}[ht]
\centering
\includegraphics[width=0.8\textwidth]{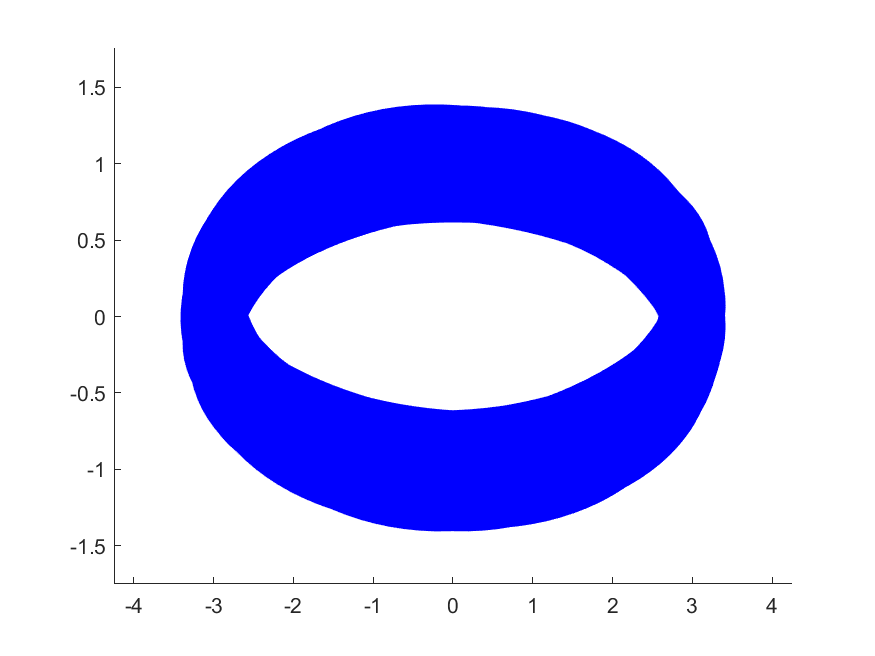}
\caption{Guaranteed velocity set $\cV^\cG_x$, with $x=(1,0)$, $L_f=0.1$, $L_G=0.3$, $f(0)=[0\textrm{ }0]^T$, and $G(0)=\mathrm{diag}(3,1)$, is given as the area in white surrounded by the blue area. The blue area and $\cV^\cG_{x}$ together provide $\cup_{(\hat{f},\hat{G})\in\cD_{con}}\hat{f}(x)+\hat{G}(x)\cU$, i.e., the \textit{optimistic} set of velocities that are available for at least one control system $\cM(\hat{f},\hat{G})$. In line with the difficulty of analytically computing them, these sets, and sets in most subsequent figures, have been approximated by Monte Carlo methods.}
\label{intell}
\end{figure} 

Our next step is thus to provide a simply computable underapproximation of $\cV^\cG_x$. Indeed, as set $\cD_{con}$ is a Cartesian product of balls and $\cU$ is a ball, set $\cV^\cG_x$ can be underapproximated in a geometrically appealing fashion. We derive such an approximation using the following results.

\vskip 5pt

\begin{lemma}
\label{lem1}
Let $a\in\RR^n$, $B\in GL(n)$, and $\cU=\BB^n(0;1)$. Then the following holds: 

\begin{enumerate}[(i)]
\item $a+B\cU\supseteq \BB^{n}(a;\|B^{-1}\|^{-1})$,
\item $\BB^{n}(a;\|B^{-1}\|^{-1})$ is a ball of maximal radius contained in $a+B\cU$.
\end{enumerate}
\end{lemma}
\begin{proof}
For (i), it clearly suffices to prove that, for every vector $\delta$, $\|\delta\|\leq \|B^{-1}\|^{-1}$, there exists $u\in\RR^n$, $\|u\|\leq 1$, such that
\begin{equation}
\label{aux1}
a+\delta=a+Bu\textrm{.}
\end{equation}
Since $B$ is invertible, setting $u=B^{-1}\delta$ satisfies \eqref{aux1}. Additionally, $\|u\|=\|B^{-1}\delta\|\leq\|\delta\|\|B^{-1}\|$ by the definition of the matrix $2$-norm. Since $\|\delta\|\leq \|B^{-1}\|^{-1}$, we have $\|u\|\leq 1$ as desired.

For (ii), we use a standard description of the ellipsoid $a+B\cU$ using singular values of $B$. In order to not drift from our paper's central narrative, we omit a detailed discussion of singular values and direct the reader to \cite{SirHor12,BenGre03}. Combining the results in the above works, we note that the length of shortest principal semi-axes of the ellipsoid $a+B\cU$ equals the smallest singular value of $B$, which equals exactly $\|B^{-1}\|^{-1}$. Thus, through a change of coordinates, we can assume without loss of generality that $a+B\cU$ is given by $\{x~|~\alpha_1x_1^2+\alpha_2x_1^2+\ldots+\alpha_nx_n^2\leq 1\}$, where $0<\alpha_1\leq\alpha_2\leq\cdots\leq\alpha_n=\|B^{-1}\|^2$. 

Consider now a ball of radius $\delta>\|B^{-1}\|^{-1}$ around any point $x'=(x'_1,\ldots,x'_n)\in\RR^n$. Then, at least one of the following inequalities hold: $x'_n+\delta>\|B^{-1}\|^{-1}$ or $x'_n-\delta<-\|B^{-1}\|^{-1}$. Hence, $\BB^n(x';\delta)\ni x'+\delta e_n\not\in a+B\cU$ or $\BB^n(x';\delta)\ni x'-\delta e_n\not\in a+B\cU$.
\end{proof}

Lemma~\ref{lem1} provides the maximal underapproximation of $\cV_x$ by a ball $\BB^n(f(x);\|G(x)^{-1}\|^{-1})$, assuming that $G(x)$ is invertible;  Figure~\ref{illem1} provides an illustration. Hence, by \eqref{vgx} and \eqref{dcon2}, $\cV^\cG_x$ can be underapproximated by 
\begin{equation}
\label{add1}
\bigcap_{\substack{\hat{a}\in\BB^n(f(0);L_f\|x\|) \\
\hat{B}\in\BB^{n\times n}(G(0);L_G\|x\|)}} \BB^n(\hat{a};\|\hat{B}^{-1}\|^{-1})\textrm{.}
\end{equation} The next two results build on Lemma~\ref{lem1} to provide a simpler approximation of $\cV^\cG_x$.

\begin{figure}[ht]
\centering
\includegraphics[width=0.8\textwidth]{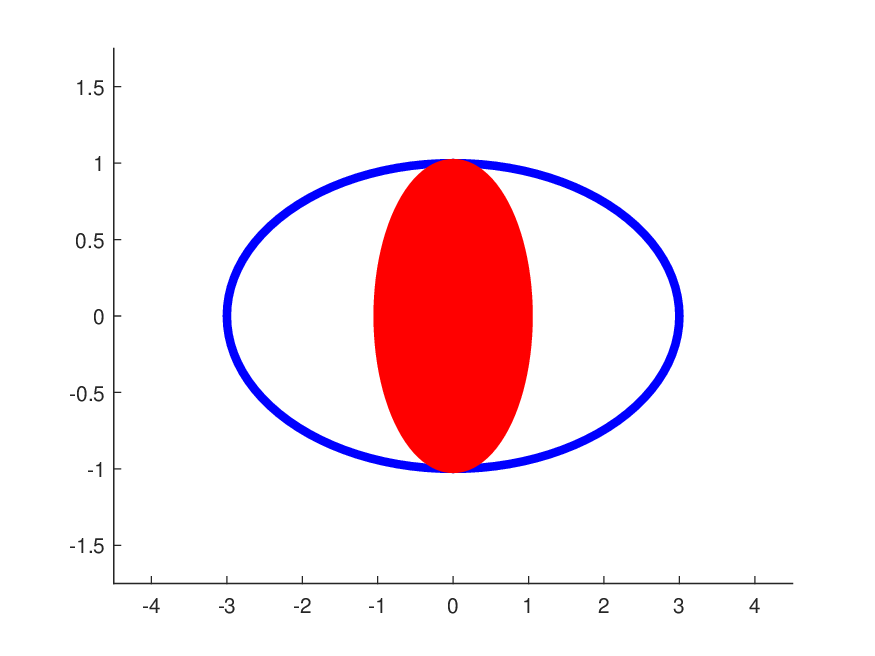}
\caption{Set $f(0)+G(0)\cU$, bounded in blue, and its underapproximation $\BB^n(f(0);\|G(0)^{-1}\|^{-1})$, drawn in red, for $f(0)=[0\textrm{ }0]^T$ and $G(0)=\mathrm{diag}(3,1)$. The red circle does not appear like a circle in order to keep the aspect ratio the same across all figures in this section, and all the illustrations legible.}
\label{illem1}
\end{figure} 

\vskip 5pt

\begin{lemma}
\label{lem2}
Let $a\in\RR^n$, $B\in GL(n)$, $0\leq r<\|B^{-1}\|^{-1}$, and $\cU=\BB^n(0;1)$. Define $\varepsilon=\min_{\hat{B}\in\BB^{n\times n}(B;r)}\|\hat{B}^{-1}\|^{-1}$. Then $$\bigcap_{\hat{B}\in\BB^{n\times n}(B;r)}a+\hat{B}\cU\supseteq \BB^{n}\left(a;\varepsilon\right)\textrm{.}$$
\end{lemma}
\begin{proof}
Assume first that all matrices $\hat{B}\in\BB^{n\times n}(B;r)$ are invertible. For all such matrices $\hat{B}$, the inequality $\|\hat{B}^{-1}\|^{-1}\geq\min_{\hat{B}\in\BB^{n\times n}(B;r)}\|\hat{B}^{-1}\|^{-1}=\varepsilon$ obviously holds, so the claim holds by Lemma~\ref{lem1}. 

It thus remains to note that all matrices $\hat{B}\in\BB^{n\times n}(B;r)$ are indeed invertible. Such a property follows directly from the Eckart-Young-Mirski theorem; we again omit the details and point the reader to \cite{BenGre03} and \cite{Bjo14} for a longer discussion.
\end{proof}

Lemma~\ref{lem2} dealt with the intersection of $a+\hat{B}\cU$ for different $\hat{B}$, while vector $a$ was fixed; Figure~\ref{illem2} provides an illustration. Let us now consider a varying vector $a$.

\begin{figure}[ht]
\centering
\includegraphics[width=0.8\textwidth]{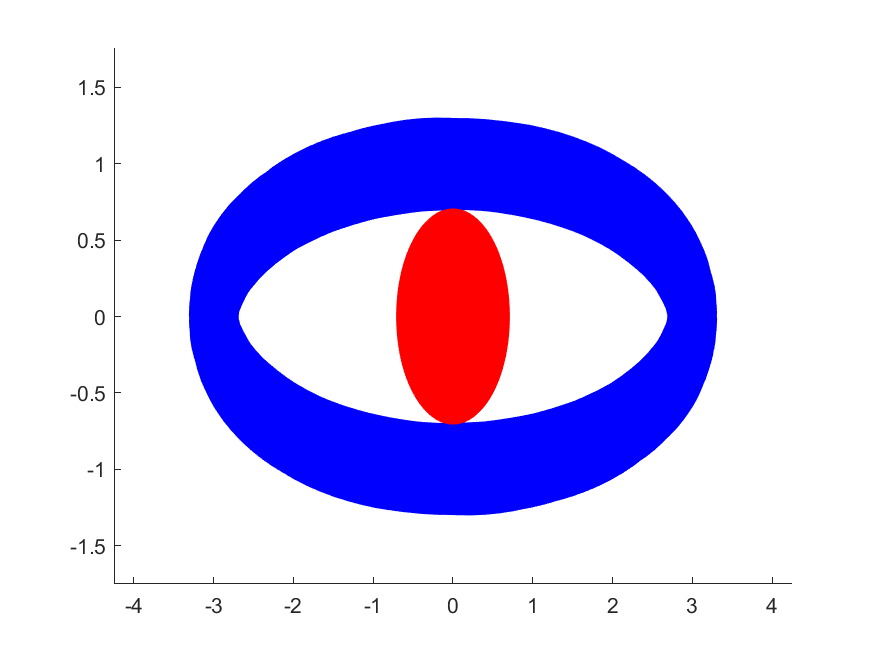}
\caption{Intersection of sets $f(0)+\hat{B}\cU$, $\hat{B}\in\BB^{n\times n}(G(0);L_G\|x\|)$, with $x=(1,0)$, $L_G=0.3$, $f(0)=[0\textrm{ }0]^T$, and $G(0)=\mathrm{diag}(3,1)$, is given as the area in white surrounded by the blue area. Its underapproximation $\BB^n(f(0);\min_{\hat{B}\in\BB^{n\times n}(G(0);L_G\|x\|)}\|\hat{B}^{-1}\|^{-1})=\BB^n(0;0.7)$ is drawn in red.}
\label{illem2}
\end{figure} 

\vskip 5pt

\begin{lemma}
\label{lem3}
Let $a\in\RR^n$, $r\geq 0$, and $R\geq r$. Then $$\bigcap_{\hat{a}\in\BB^{n}(a;r)}\BB^{n}(\hat{a};R)=\BB^{n}(a;R-r)\textrm{.}$$
\end{lemma}
\begin{proof}
Let $v\in\BB^{n}(a;R-r)$. Then, $\|v-a\|\leq R-r$. Thus, for every $\hat{a}\in\BB^{n}(a;r)$, $\|v-\hat{a}\|\leq\|v-a\|+\|a-\hat{a}\|\leq R-r+r\leq R$. Hence, $v\in\BB^n(\hat{a};R)$.

Now, let $v\not\in\BB^{n}(a;R-r)$. Hence, $\|v-a\|>R-r$. Now, choose $\hat{a}=a+r(a-v)/\|a-v\|$. Clearly, $\hat{a}\in\BB^{n}(a;r)$. On the other hand, $\|v-\hat{a}\|=\|v-a-r(a-v)/\|a-v\|\|=\|(a-v)\|(1+r/\|a-v\|)=r+\|a-v\|>R$. Thus, $v\not\in\BB^n(\hat{a};R)$.
\end{proof}

Finally, we can combine the results of the above lemmas to obtain an approximation for $\cV^\cG_x$.

\vskip 5pt

\begin{theorem}
\label{approvx}
Let $\cU$, $L_f$, $L_G$, $f(0)$, and $G(0)$ be as above. Let $x\in\RR^n$ satisfy the inequality $(L_f+L_G)\|x\|\leq\|G(0)^{-1}\|^{-1}$. Define $$\overline\cV^\cG_x=\BB^n\left(f(0);\|G(0)^{-1}\|^{-1}-L_f\|x\|-L_G\|x\|\right)\textrm{.}$$ Then, $\overline\cV^\cG_x\subseteq\cV^\cG_x\textrm{.}$
\end{theorem}
\begin{proof}
As noted in \eqref{add1},
$$\cV^\cG_x=\bigcap_{\hat{a}\in\BB^n(f(0);L_f\|x\|)}\left(\bigcap_{\hat{B}\in\BB^{n\times n}(G(0);L_G\|x\|)}\hat{a}+\hat{B}\cU\right)\textrm{.}$$ Then, by Lemma~\ref{lem2}, $$\cV^\cG_x\supseteq \bigcap_{\hat{a}\in\BB^n(f(0);L_f\|x\|)}\BB^{n}\left(\hat{a};\min_{\hat{B}\in\BB^{n\times n}(G(0);L_G\|x\|)}\|\hat{B}^{-1}\|^{-1}\right)\textrm{.}$$

By Weyl's inequality for singular values \citep{Ste90} and a characterization of $\|\hat{B}^{-1}\|^{-1}$ as an appropriate singular value of matrix $\hat{B}$ \citep{BenGre03}, $\|G(0)^{-1}\|^{-1}-r\leq \min_{\hat{B}\in\BB^{n\times n}(G(0);r)}\|\hat{B}^{-1}\|^{-1}$ for any $r\geq 0$. Hence,
$$\cV^\cG_x\supseteq \bigcap_{\hat{a}\in\BB^n(f(0);L_f\|x\|)}\BB^{n}(\hat{a};\|G(0)^{-1}\|^{-1}-L_G\|x\|)\textrm{.}$$
The claim of the theorem now holds by Lemma~\ref{lem3}, since we assume $L_f\|x\|\leq\|G(0)^{-1}\|^{-1}-L_G\|x\|$.
\end{proof}

Theorem~\ref{approvx} is the central result of the current section. Assuming $L_f+L_G>0$, for each $x$ in a ball around $x_0=0$ of radius $1/((L_f+L_G)\|G(0)^{-1}\|)$ it generates a nonempty set $\overline\cV^\cG_x$ of guaranteed velocities, i.e., velocities $v$ for which there certainly exist a control input $u\in\cU$ such that $f(x)+G(x)u=v$. The case of $L_f+L_G=0$ is uninteresting as it results in dynamics \eqref{consys} being entirely known from $f(0)$ and $G(0)$, and the true reachable set can thus be computed using standard methods described at the end of subsequent section.

Theorem~\ref{approvx} is illustrated by Figure~\ref{ilthm}, with the underapproximation $\overline\cV^\cG_x$ added to the illustration of $\cV^\cG_x$ from Figure~\ref{intell}. While one might expect that consecutive underapproximations in the three lemmas preceding Theorem~\ref{approvx} would result in an exceedingly bad approximation $\overline\cV^\cG_x$, the following proposition shows that $\overline\cV^\cG_x$ is indeed the best approximation of $\cV^\cG_x$ by a ball.

\begin{figure}[ht]
\centering
\includegraphics[width=0.8\textwidth]{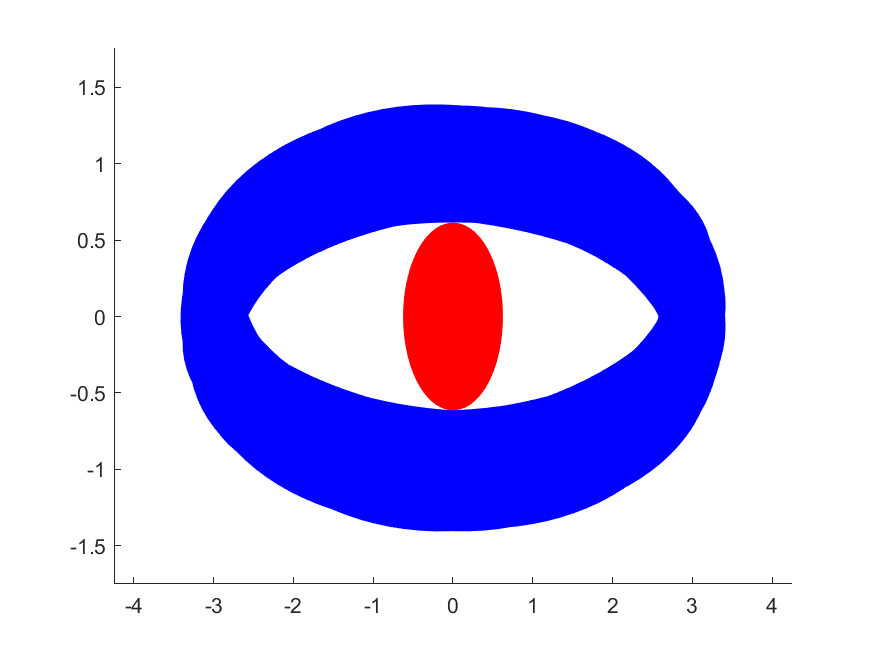}
\caption{Set $\cV^\cG_x$, with $x=(1,0)$, $L_G=0.3$, $f(0)=[0\textrm{ }0]^T$, and $G(0)=\mathrm{diag}(3,1)$, is given as the white area in center. Its underapproximation $\overline\cV^\cG_x=\BB^n(0;0.6)$ is drawn in red.}
\label{ilthm}
\end{figure} 

\vskip 5pt

\begin{proposition}
\label{bestball}
Let $\cV^\cG_x$ and $\overline\cV^\cG_x$ be as above. Then, $\overline\cV^\cG_x$ is a ball of maximal radius that is contained in $\cV^\cG_x$.
\end{proposition}
\begin{proof}
The proof follows the same steps as the proof of part (ii) of Lemma~\ref{lem1}, which proves the proposition in the special case of $L_f=0$ and $L_G=0$. As the detailed proof is technically inelegant, if straightforward, we provide its outline. 

As in part (ii) of Lemma~\ref{lem1}, we may assume without loss of generality that $f(0)+G(0)\cU$ is an ellipsoid with principal semi-axes parallel to coordinate axes $e_i$ and its shortest principal semi-axes given by $\pm\|G(0)^{-1}\|^{-1} e_n$. Now consider two ellipsoids centered around $ L_f\|x\|e_n$ and $-L_f\|x\|e_n$, respectively, such that their principal semi-axes are parallel with the coordinate axes and such their shortest principal semi-axes are given by $(\|G(0)^{-1}\|^{-1}-L_G\|x\|)e_n$. It can be shown that these ellipsoids can be written as $\hat{a}+\hat{B}\cU$, where $\hat{a}=\pm L_f\|x\|e_n$ and $\hat{B}\in\BB^{n\times n} (G(0);L_G\|x\|)$. By \eqref{dcon2}, there thus exist two control systems $\cM(\hat{f},\hat{G})$ consistent with assumed knowledge such that the constructed ellipsoids are their available velocity sets at $x$.

For any $x'\in\RR^n$ and any $\delta>\|G(0)^{-1}\|^{-1}-L_G\|x\|-L_f\|x\|$, set $\BB^n(x';\delta)$ will contain at least one point $y=(y_1,\ldots,y_n)$ with $|y_n|\geq \delta$. On the other hand, any such point is in at most one of the two above ellipsoids $\hat{a}+\hat{B}\cU$, and is thus not in their intersection, and consequently not in $\cV^\cG_x$.
\end{proof}

While $\overline\cV^\cG_x$ is the best approximation of $\cV^\cG_x$ by a ball, such an underapproximation still potentially discards a large part of the intersection. In the case $x=0$, the quality of the approximation, i.e., the volume of the set being discarded, depends on the magnitude of $\|G(0)^{-1}\|^{-1}$ (i.e., the length of shortest principal semi-axes of $f(0)+G(0)\cU$) relative to the lengths of other principal semi-axes of $\cV_0=\cV^\cG_0$. By considering solely the ratio of the length of the longest principal semi-axes and the shortest principal semi-axes, we obtain that the volume being discarded roughly depends on the \textit{condition number} of matrix $G(0)$ \citep{BenGre03}. Such a number, representing the ratio of the largest and smallest singular value of a matrix, is commonly used to determine ``distance'' from matrix singularity. In other words, the closer $G(0)$ is to being singular, i.e., the closer the control system is to being underactuated at $x_0=0$, the worse the provided underapproximation will be. 

\begin{remark}
If system $\cM(f,G)$ satisfied $\mathrm{rank}(G(0))=k<n$, i.e., Assumption~\ref{ass} did not hold, all our work could be performed by considering the smallest non-zero singular value $\sigma_k$ of $G(0)$ instead of $\|G(0)^{-1}\|^{-1}$; the set $f(0)+G(0)\cU$ would be a $k$-dimensional ellipsoid with shortest principal semi-axes of length $\sigma_k$, and we would obtain a $k$-dimensional underapproximation for $\cV^\cG_x$. In that case, the quality of the underapproximation would depend on the ratio between the largest and smallest \emph{nonzero} singular value of matrix $G(0)$.
\end{remark}

In the remainder of the text, we assume that $L_f+L_G>0$ for both technical and motivational reasons: the case $L_f=L_G=0$ results in the system dynamics being known immediately from our assumed knowledge. Its reachable set can thus be computed using the methods described in the subsequent section.

We now proceed to exploit the geometrically simple form of $\overline{\cV}^\cG_x$ to obtain a control system with known dynamics whose reachable set is an underapproximation of the GRS $R^\cG(T,0)$.

\section{Reachable Set}
\label{reachse}
Using the result of Theorem~\ref{approvx}, we can underapproximate the set of trajectories to the differential inclusion \eqref{incsys}, i.e., the controlled dynamics \eqref{consys}, by the trajectories of the differential inclusion
\begin{equation}
\label{apsys}
\begin{split}
\dot{x}\in\overline\cV^\cG_x\textrm{,} \\
x(0)=x_0\textrm{.}
\end{split}
\end{equation}
Set $\overline\cV^\cG_x$ is only defined for $x$ such that $\|G(0)^{-1}\|^{-1}-L_G\|x\|-L_f\|x\|\geq 0$, and we only care about the reachability of \eqref{apsys} within the ball $\BB_{know}=\BB^n(0;1/((L_f+L_G)\|G(0)^{-1}\|))$. The interpretation of set $\BB_{know}$ is that any state $x\not\in\BB_{know}$ is so far away from $x_0=0$ that the local dynamics at $x_0$ do not provide any information about local dynamics at $x$.

In order to exploit previous work on computing reachable sets, we continuously extend differential inclusion \eqref{apsys} to the entire $\RR^n$ by defining $\overline\cV^\cG_x=\{f(0)\}$ outside of the ball $\BB^n(0;1/((L_f+L_G)\|G(0)^{-1}\|))$. We emphasize that such an extension is purely technical, i.e., $f(0)$ is \textit{not} a guaranteed velocity at states $x\not\in\BB_{know}$.

Let $\overline{R}(T,x_0)$ be the reachable set of \eqref{apsys} at time $T$, i.e., the set of all states achieved by trajectories $\phi:[0,+\infty)\to\RR^n$ that satisfy $\phi(0)=x_0$ and \eqref{apsys} for all $t\leq T$. It can be shown that any solution of \eqref{apsys} connecting two points in $\BB_{know}$ needs to entirely lie within $\BB_{know}$. Thus, underapproximation $\overline{\overline{R}}(T,x_0)$ of the guaranteed reachable set $R^\cG(T,x_0)$ is given by \begin{equation}
\label{olol}
\overline{\overline{R}}(T,x_0)=\overline{R}(T,x_0)\cap\BB_{know}\subseteq R^\cG(T,x_0)\textrm{.}
\end{equation}

The question that remains to be considered is, thus, determining the set $\overline{R}(T,x_0)$. The problem of computing reachable sets of differential inclusions is generally difficult \citep{Asaetal01, QuiVel02}. However, we will show that the fact that each set $\overline\cV^\cG_x$ is a ball allows us to describe the set $\overline{R}(T,x_0)$ as a solution to a partial differential equation given in closed form. 

Going back to inclusion \eqref{apsys}, we note that --- analogously to our conversion of \eqref{consys} into inclusion \eqref{incsys} --- inclusion \eqref{apsys} can be interpreted as, now \textit{entirely known}, ordinary differential equation
\begin{equation}
\label{newconsys}
\begin{split}
\dot{x}=a+ & g(\|x\|)u\textrm{,} \\
x(0) & =0\textrm{,}
\end{split}
\end{equation}
with $a=f(0)$, $u\in\cU=\BB^n(0;1)$, and
where $g:[0,+\infty)\to\RR$ is a ramp function given by $g(s)=
\|G(0)^{-1}\|^{-1}-(L_G+L_f)s$ if $s\leq \|G(0)^{-1}\|^{-1}/(L_G+L_f)$ and $g(s)=0$ otherwise.

The set $\overline{R}(T,x_0)$ is now the reachable set of the system with dynamics \eqref{newconsys}, in the sense discussed in Section \ref{probst}. As mentioned in the preliminary sections, there exist multiple approaches to approximating or computing reachable sets for nonlinear control systems \citep{Che17,Vin80,ScoBar13,Baietal13}; we follow the Hamilton-Jacobi approach described by \cite{Che17} and references therein. 

The approach in \cite{Che17} relies on defining any sufficiently smooth function $l:\RR^n\to\RR$ with $\{0\}=\{x~|~l(x)\leq 0\}$. Applied to our setting and switching between backward reachable sets described by \cite{Che17} and forward reachability that is of interest to us, the reachable set $\overline{R}(T,x_0)$ is then given by $$\overline{R}(T,x_0)=\{x\in\RR^n~|~V(T,x)\leq 0\}\textrm{.}$$ Function $V(t,x)$ is the viscosity solution \citep{CraLio83} of the \textit{Hamilton-Jacobi partial differential equation}
\begin{equation}
\label{pde}
\begin{split}
& V_t(t,x)+H(x,V_x(t,x))=0\textrm{ for all } t\in[-T,0], x\in\RR^n\textrm{,} \\
& V(0,x)=l(x) \textrm{ for all } x\in\RR^n\textrm{,}
\end{split}
\end{equation}
where $H(x,\lambda)=\min_{u\in\cU}\lambda^T(-a+g(\|x\|)u)$.
Due to the simplicity of the dynamics \eqref{newconsys}, the Hamiltonian $H(x,\lambda)$ can be written in closed form. The following theorem presents such a result.

\vskip 5pt
\begin{theorem}
\label{thm2}
Let $l:\RR^n\to\RR$ with $\{x~|~l(x)\leq 0\}=\{0\}$ and $T>0$. Then, $$\overline{\overline{R}}(T,x_0)=\{x\in\BB_{know}~|~V(T,x)\leq 0\}\textrm{,}$$ where
$V$ is the viscosity solution of 
\begin{equation}
\label{pde2}
\begin{split}
& V_t(t,x)-V_x(t,x)^Ta-\|V_x(t,x)\|g(\|x\|)=0 \textrm{ for all } t\in[-T,0], x\in\RR^n\textrm{,} \\
& V(0,x)=l(x) \textrm{ for all } x\in\RR^n\textrm{.}
\end{split}
\end{equation}
\end{theorem}
\begin{proof}
The result follows simply from \eqref{olol} and \eqref{pde}, noticing that $H(x,\lambda)=\min_{u\in\cU}\lambda^T(-a+g(\|x\|)u)=-\lambda^Ta+g(\|x\|)\min_{u\in\cU}\lambda^Tu=-\lambda^Ta-g(\|x\|)\lambda^T\lambda/\|\lambda\|=-\lambda^Ta-g(\|x\|)\|\lambda\|$. 
\end{proof}

Solving \eqref{pde2} thus yields the desired underapproximation $\overline{\overline{R}}(T,x_0)$ of the set $R^\cG(T,x_0)$. We remark that the above method for computing $\overline{\overline{R}}(T,x_0)$ is not the only possible one; for instance, as $x\mapsto\overline\cV^\cG_x$ easily satisfies the one-sided Lipschitz and upper semicontinuity conditions of \cite{Donetal03}, it is possible to use the exponential formula derived therein to obtain an expression for $\overline{\overline{R}}(T,x_0)$ as a limit of sets, with a guaranteed order of convergence. 

Having found $\overline{\overline{R}}(T,x_0)$, let us briefly return to the question of determining a control signal that drives the system satisfying the original dynamics \eqref{consys} from $x_0$ to an element of this set, first raised in Section \ref{probst}. Such a control signal clearly depends on system dynamics. However, the results of \cite{Che17} provide an expression for a control signal $\overline{u}$ on \eqref{newconsys} that drives the system from $x_0$ to the desired objective. By taking $v=a+g(\|x\|)\overline{u}$, such a signal yields a function of desired velocity vectors $v(t)$ that will drive the system to the desired objective. As discussed in the proof of Lemma~\ref{lem2}, $G(x)$ is invertible for all $x\in\BB_{know}$. Thus, if $f(x(t))$ and $G(x(t))$ were known, the appropriate control input $u(r)$ for \eqref{consys} would be given by $$u(t)=\left(G
\left(x(t)\right)\right)^{-1}\left(v(t)-f\left(x(t)\right)\right)\textrm{.}$$ Combining $v$ with the algorithm of \cite{Ornetal17,Ornetal19} --- technical assumptions notwithstanding --- to determine $f(x(t))$ and $G(x(t))$ at time $t$ with arbitrarily small error, we can thus obtain a control input that takes the system from $x_0$ to any desired state in $\overline{\overline{R}}(T,x_0)$.

We now proceed to illustrate the results obtained in this section and preceding sections by way of numerical examples.

\section{Examples}
\label{exa}

We investigate two scenarios. In the first one, computing the GRS of a 1D system will validate our theory and show that, while our current theoretical results do not establish equality between $\overline{\overline{R}}(T,0)$ and $R^\cG(T,0)$, those two sets may indeed coincide. Thus, determining conditions for their equality, or at least the equality of $R^\cG(T,0)$ and the reachable set of inclusion $\dot{x}\in\cV^\cG_x$, is a meaningful direction for future work.

The second example models the motivational scenario of this paper's narrative. We will model an aircraft with deteriorating actuation capabilities and use our theory --- slightly modified to allow for time-varying dynamics --- to determine the set of states that it is guaranteed to safely reach. 

\subsection{Autoregulation}
\label{exa1}

We will first illustrate the developed results on a simple single-dimensional system motivated by a systems biology narrative. The system state is given by $x\in\RR$, and the input is given by $u\in[-1,1]$. The available knowledge of the system dynamics is that they satisfy \eqref{consys}, with $f(0)=0$, $G(0)=1/2$, $L_f=0.1$, and $L_G=0.7$. The true --- but unknown --- dynamics of the system are given by \begin{equation}
\label{hill}
\begin{split}
\dot{x}(t) & =\frac{(x(t)+1)^2}{1+(x(t)+1)^2}u\textrm{,} \\
x(0) & =0\textrm{.}
\end{split}
\end{equation}

Dynamics \eqref{hill} are a slight modification of dynamics describing the phenomenon of gene autoregulation --- the ability of a gene product to determine its own rate of production by binding to those elements of the gene that regulate its production \citep{Alo06b,Jabetal92}. However, in the context of this paper, our interest in equation \eqref{hill} is purely mathematical.

Based on the above available knowledge of the system dynamics, we seek to determine the underapproximation $\overline{\overline{R}}(T,0)$ from \eqref{olol}, and compare it to the GRS $R^\cG(T,0)$. We will also compare it to the true reachability set $R(T,0)$; however, the latter set would be impossible to determine solely from available knowledge. 

With true dynamics given in equation \eqref{hill}, we note that bounds $L_f=0.1$ and $L_G=0.7$ are loose; for instance, the maximal rate of change of $f$ equals $0$ (i.e., $f(x)=0$ for all $x$). Consequently, the guaranteed reachable set is smaller than the one that would have been obtained with tighter bounds $L_f$ and $L_G$.

In order to determine $\overline{\overline{R}}(T,0)$, we follow Theorem~\ref{approvx} and the work in Section \ref{reachse}. Namely, from Theorem~\ref{approvx}, by plugging in relevant values of $f(0)$, $G(0)$, $L_f$ and $L_G$, we obtain $\overline{\cV}^\cG_x=\BB^1(0;0.5-0.8\|x\|)=[-0.5+0.8|x|,0.5-0.8|x|]$ for $|x|\leq 5/8$. By \eqref{olol} and \eqref{newconsys}, the set $\overline{\overline{R}}(T,0)$ is thus given by intersecting $[-5/8,5/8]$ with the reachable set of \begin{equation}
\label{ex1}
\begin{split}
&\dot{x} =(0.5-0.8|x|)u\textrm{,} \\
& x(0) =0,
\end{split}
\end{equation} 
with $u\in[-1,1]$.

The solution to \eqref{ex1} can be found analytically for any $u$; it can consequently be easily shown that the reachable set $\overline{R}(T,0)$ of \eqref{ex1} equals $$\overline{R}(T,0)=\left[\frac{5}{8}(e^{-0.8T}-1),\frac{5}{8}(1-e^{-0.8T})]\right]\textrm{.}$$ Hence, $\overline{R}(T,0)\subseteq [-5/8,5/8]$ and thus $\overline{\overline{R}}(T,0)=\overline{R}(T,0)$.

By the construction of $\overline{\overline{R}}(T,0)$, the GRS satisfies $R^\cG(T,0)\supseteq\overline{\overline{R}}(T,0)$. We will now show $R^\cG(T,0)\subseteq\overline{\overline{R}}(T,0)$, i.e., $R^\cG(T,0)=\overline{\overline{R}}(T,0)$.

Consider first the dynamics $\dot{x}=0.1|x|+(0.5-0.7|x|)u$, $x(0)=0$. These dynamics are consistent with the available knowledge about the system, i.e., if $\hat{f}(x)=0.1|x|$ and $\hat{G}(x)=0.5-0.7|x|$, then $(\hat{f},\hat{G})\in\cD_{con}$. By considering the dynamics when $x\leq 0$, the reachable set $R^{\hat f,\hat G}(T,0)$ can be shown to equal $[\frac{5}{8}(e^{-0.8T}-1),\alpha_T]$ for some $\alpha_T\geq 0$.

We can analyze dynamics $\dot{x}=-0.1|x|+(0.5-0.7|x|)u$ analogously. By looking at the case of $x\geq 0$, we obtain that the reachable set of a system with these dynamics equals $[\beta_T,\frac{5}{8}(1-e^{-0.8T})]$ for some $\beta_T\leq 0$.

Hence, by \eqref{grs}, $R^\cG(T,0)\subseteq\overline{\overline{R}}(T,0)$, and thus $R^\cG(T,0)=\overline{\overline{R}}(T,0)$. In other words, our underapproximation is not an approximation at all and exactly equals the GRS. While establishing further theoretical results on equality of the two sets is a topic for future work, the above fact encourages us to believe that $\overline{\overline{R}}(T,0)$ --- or, more likely, the reachable set of inclusion \eqref{vdifinc} --- may indeed equal the GRS.

Finally, we note that the true reachable set $R(T,0)$ of dynamics \eqref{hill} equals $$R(T,0)=\left[\frac{\sqrt{T^2+4}-T-2}{2},\frac{\sqrt{T^2+4}+T-2}{2}\right]\textrm{.}$$ While such a set is impossible to find only from the available knowledge of $f(0)$, $G(0)$, $L_f$, and $L_g$, it is indeed not difficult --- if algebraically onerous --- to verify that $R^\cG(T,0)\subseteq R(T,0)$ for all $T\geq 0$. Additionally, the relative difference between $R^\cG(T,0)$ and $R(T,0)$ vanishes as $T\to 0$: formally, if $l(\cI)$ represents the length of interval $\cI$, then $$\lim_{T\to 0}\frac{l(R^\cG(T,0))}{l(R(T,0))}=\lim_{T\to 0}\frac{5(1-e^{-0.8T})/4}{T}=1\textrm{.}$$ This feature is in line with our intuition that, the closer the system state is to $x_0=0$, the better we can predict its dynamics. We now proceed to a slightly more involved example.

\subsection{Aircraft Diversion}
\label{exa2}

In this section, we will consider a rudimentary model of an aircraft exposed to unexpected actuator deterioration and environmental effects on dynamics (e.g., wind), and determine a set of states (``diversion airports'') that are guaranteed to be reachable by the aircraft.

We assume that the aircraft dynamics are known to satisfy the following equation: 
\begin{equation}
\label{dynac}
\begin{split}
& \dot{x}=f(t)+G(t)u\textrm{,} \\
& x(0) =0\textrm{,} \\
\end{split}
\end{equation}
where $x\in\RR^2$ and $u\in\cU=\BB^2(0;1)$. Function $f$ models the time-dependent drift caused by the environment, while $G$ models the actuator deterioration. We will use $G(t)=1-t/10$ for $t\leq 10$ and $G(t)=0$ for $t>10$. In other words, the vehicle's actuators deteriorate linearly until they entirely stop working at time $t=10$. We choose $f(t)=[0.1(\cos(t)+1)\textrm{ }0]^T$ to model the effect of a non-zero drift on the GRS --- the chosen $f(t)$ represents the wind flowing in a constant direction with varying strength. We emphasize that those functions, apart from the values of $f(0)$ and $G(0)$, are considered unknown when determining the GRS.

While it is clear that model \eqref{dynac} does not accurately represent the dynamics of any existing aircraft and the described environmental dynamics are simplistic, versions of the underlying single integrator model $\dot{x}=u$ have been extensively used in planning, e.g., by \cite{Bervan91,OhAhn14,Niketal16}.

We note that dynamics \eqref{dynac} do not formally fall within the class of dynamics given by \eqref{consys}; however, we can bring them into that class by the standard method of appending a variable $t$ to the state vector $x=(x_1,x_2)\in\RR^2$ as a state $x_3$ with $\dot{x}_3=1$, $x_3(0)=0$. Such a method has been used in similar circumstances, e.g., by \cite{Ornetal17}. While such modified dynamics still slightly differ from the assumptions of this paper --- namely, the system has three states, but only two inputs, functions $f$ and $G$ are known to depend on only one state, and the dynamics of one of the states are fully known --- the theory of previous sections can be directly adapted to the new circumstances. We can thus proceed with dynamics \eqref{dynac}. Our interest is in determining the set of all states that can eventually be reached, i.e., the \textit{guaranteed eventual reachability set} $\cup_{T\geq 0} R^\cG(T,0)$. We assume that the following knowledge is available: $f(0)=[0.2\textrm{ }0]^T$, $G(0)=1$, $L_f=0.1$, and $L_G=1/5$. As in the previous example, such knowledge is pessimistic: while the bound on the rate of change of drift is correct, the maximal deterioration rate of the actuators is assumed to be twice as large as the true rate.

By adapting Theorem~\ref{approvx} to the framework of \eqref{dynac}, we obtain a subset of the guaranteed velocity set at time $t$ equaling $\overline{\cV}^\cG_t=\BB^2((0.2,0);1-0.3t)$, with $t\leq 10/3$. At time $t\geq 10/3$ the prior information does not allow us to know anything about the effect of actuators on the system.

Analogously to the work of Section~\ref{reachse}, from $\overline{\cV}^\cG_t=\BB^2((0.2,0);1-0.3t)$ we obtain differential equation 
\begin{equation}
\label{dynac2}
\dot{x}=a+(1-0.3t)u\textrm{,}
\end{equation} where $a=[0.2\textrm{ }0]^T$ with a reachable set denoted by $\overline{\overline{R}}(T,0)$. We are interested in determining the set $\overline{\overline{R}}(0)$ defined by $$\overline{\overline{R}}(0)=\bigcup_{T\in [0,10/3]} \overline{\overline{R}}(T,0)\subseteq \bigcup_{T\in [0,10/3]}R^\cG(T,0)\textrm{.}$$

By employing the substitution $z(t)=x-at$, equation \eqref{dynac2} becomes $\dot{z}=(1-0.3t)u$. Its reachable set at time $T$ can be easily computed to equal $\BB^2(0;T-3T^2/20)$. Thus, $\overline{\overline{R}}(T,0)=\BB^2(aT;T-3T^2/20)$. Approximation $\overline{\overline{R}}(0)$ of the set of guaranteed eventual reachability is compared with true reachable sets $\cup_{T\in[0,10/3]}{R}(T,0)$, $\cup_{T\in[0,10]}{R}(T,0)$, and $\cup_{T\geq 0}{R}(T,0)$ in Figure~\ref{illdynac}; the last three sets can be computed from \eqref{dynac} by making the substitution $z=x-[0.1(t+\sin(t))\textrm{ }0]^T$ and integrating both sides of the resulting equation.

\begin{figure}[ht]
\centering
\includegraphics[width=0.8\textwidth]{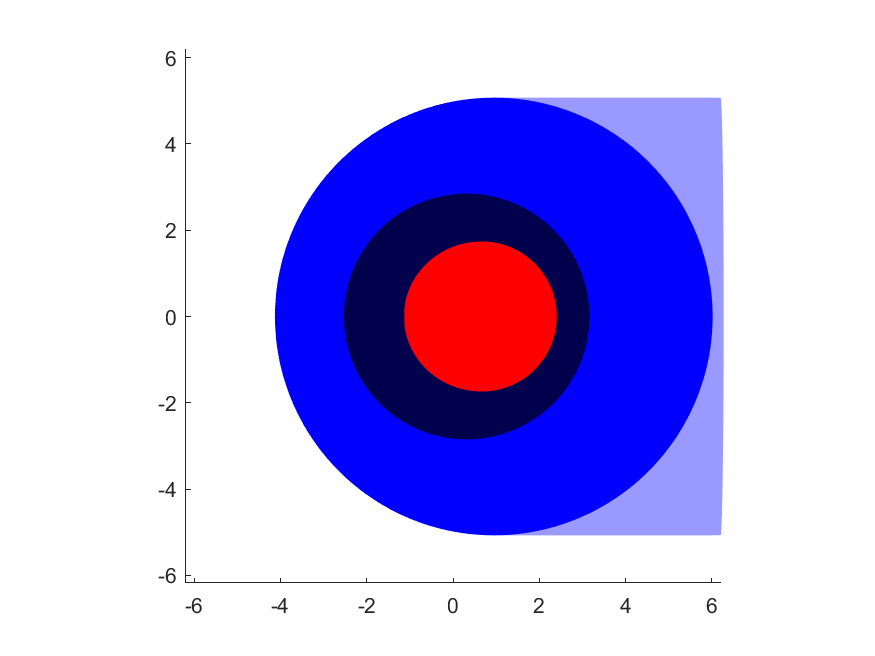}
\caption{Set $\overline{\overline{R}}(0)=\cup_{T\in [0,10/3]} \overline{\overline{R}}(T,0)$ is drawn in red. The sets of true reachable states, $\cup_{T\in[0,10/3]}{R}(T,0)$, $\cup_{T\in[0,10]}{R}(T,0)$, and $\cup_{T\geq 0}{R}(T,0)$ are drawn in increasingly light shades of blue. We note that set $\cup_{T\geq 0}{R}(T,0)$ is unbounded as the vehicle can continue ``gliding'' even after losing all actuation capabilities at time $t=10$.}
\label{illdynac}
\end{figure} 

As shown in Figure~\ref{illdynac}, our results indeed provide an underapproximation of the set of states that the aircraft can reach. While the true reachable set is larger than $\overline{\overline{R}}(0)$, such a property naturally follows from the lack of available accurate knowledge of true dynamics; only the dynamics at time $t=0$ are known, and the rate of change of dynamics is assumed to be larger than the true rate of change.

\section{Conclusions and Future Work}
\label{conc}

Motivated by the problem of choosing an appropriate desired state for a control system with unknown dynamics, this paper establishes preliminary results on estimating the GRS --- a set of states that are guaranteed to be reachable by such a system, given our current knowledge of the system. The mathematical framework of the paper follows the work of \cite{Ornetal17,Ornetal19}; our paper exploits information on local system dynamics that can be collected using an algorithm proposed therein to obtain an underapproximation of the GRS.

The results of this paper represent an initial effort in estimating the GRS. Notably, while we show that the produced set $\overline{\overline{R}}(T,x_0)$ is indeed its subset, we provide only a preliminary discussion of the quality of such an underapproximation. 

The difference between the GRS and $\overline{\overline{R}}(T,x_0)$ depends on the quality of two intermediate approximations: (i) the approximation of the GRS by the reachable set of the differential inclusion $\dot{x}\in\cV^\cG_x$, where $\cV^\cG_x$ are the velocities guaranteed to be available at state $x$, and (ii) the approximation of $\cV^\cG_x$ by a ball $\overline\cV^\cG_x$. In (i), the current paper only shows that the latter set is a subset of the former. However, the numerical results of Section \ref{exa1}, along with the intuition described in Section \ref{gvel}, strongly imply that the two sets may be equal, potentially under some additional conditions. Proving such a claim would significantly improve the theoretical underpinning of the estimation method provided by this paper.

In (ii), the current paper shows that $\overline\cV^\cG_x$ is a ball with maximal radius of all the balls contained in $\cV^\cG_x$. Nonetheless, as discussed at the end of Section \ref{gvel}, $\overline\cV^\cG_x$ may have a significantly smaller volume than $\cV^\cG_x$, thus resulting in an underapproximation of the GRS of significantly lower volume than the true set. Approximating $\cV^\cG_x$ by a geometrical object that more closely fits the complex shape of $\cV^\cG_x$ would yield a better approximation of the GRS. A possible candidate, potentially simple enough to enable computational work, is a general ellipsoid considered for approximation of related sets by \cite{KurVar06b}.

Finally, let us move from the problem of computing the GRS, \textit{given} the assumed system knowledge, to the problem of obtaining a GRS that is closer to the true reachable set of a control system. In that vein, a possible area for theoretical improvement of the results in the paper is in considering a different class of available prior knowledge of system dynamics. While we currently consider and exploit only local dynamics at a single point, the algorithm proposed in \cite{Ornetal17,Ornetal19} can produce local dynamics at arbitrarily many points on a single trajectory. Possibly combined with additional prior knowledge about system dynamics (e.g., bounds on higher partial derivatives of $f$ and $G$), exploiting such information will result in larger sets of guaranteed velocities compared to current work, and thus in guaranteed reachability sets that are closer underapproximations of the true reachable sets. The primary obstacle to successfully improving the estimates using such additional information, even if it may be naturally available, is computational. Namely, with a large number of constraints describing all the information available about the system dynamics, set $\{(\hat{f}(x),\hat{G}(x))~|~ (\hat{f},\hat{G})\in\cD_{con}\}$, crucial in approximating $\cV^\cG_x$, may be geometrically more complex than the one described in this paper.

\section*{Acknowledgments}
This work was supported by an Early Stage Innovations grant from NASA's Space Technology Research Grants Program, grant no.\ 80NSSC19K0209. The author wishes to thank Ufuk Topcu and Franck Djeumou for discussions on related topics.

\bibliographystyle{plainnat}
\bibliography{refs}
\end{document}